\documentclass[11pt,twoside]{amsart}
\usepackage{amsmath}
\usepackage{amssymb}
\usepackage{amsthm}
\usepackage{latexsym}
\usepackage{amscd}
\usepackage{xypic}
\xyoption{curve}
\usepackage{ifthen} 
\usepackage{hyperref} 
\usepackage{graphicx}
\usepackage{enumerate}

\usepackage{epsfig}
\usepackage{caption}

\usepackage{enumitem}

 \def\cocoa{{\hbox{\rm C\kern-.13em o\kern-.07em C\kern-.13em o\kern-.15em A}}}

\usepackage{xcolor}

\newtheorem{theorem}{Theorem}[section]
\newtheorem{lemma}[theorem]{Lemma}
\newtheorem{proposition}[theorem]{Proposition}
\newtheorem{corollary}[theorem]{Corollary}

\theoremstyle{definition}
\newtheorem{remark}[theorem]{Remark}
\newtheorem{definition}[theorem]{Definition}
\newtheorem{example}[theorem]{Example}

\newcommand {\Hom}{\mathcal{H}\kern -0.25ex{\mathit om}}
\newcommand {\Spl}{\mathcal{S}\kern -0.25ex{\mathit pl}}
\newcommand {\Ext}{\mathcal{E}\kern -0.25ex{\mathit xt}}

\newcommand {\rk}{\mathrm{rk}}

\newcommand {\Hilb}{\mathcal{H}\kern -0.25ex{\mathit ilb\/}}

\newcommand {\bZ}{\mathbb{Z}}

\newcommand {\bP}{\mathbb{P}}

\newcommand{\cE}{{\mathcal E}}
\newcommand{\cF}{{\mathcal F}}

\newcommand{\cO}{{\mathcal O}}
\newcommand{\cG}{{\mathcal G}}
\newcommand{\cI}{{\mathcal I}}

\newcommand{\Pic}{\operatorname{Pic}}

\def\p#1{{\bP^{#1}}}

\def\mapright#1{\mathbin{\smash{\mathop{\longrightarrow}
\limits^{#1}}}}

\title[Rank two aCM bundles on general determinantal quartic surfaces in $\p3$]{Rank two aCM bundles \\ on general determinantal quartic surfaces in $\p3$}



\author[Gianfranco Casnati]{Gianfranco Casnati}
\thanks{The author is member of the GNSAGA group of INdAM}

\begin{document}

\dedicatory{This paper is dedicated to the memory of A.T. Lascu.}

\begin{abstract}
Let $F\subseteq\p 3$ be a smooth determinantal quartic surface which is general in the N\"other--Lefschetz sense. In the present paper we give a complete classification of locally free sheaves $\cE$ of rank $2$ on $F$
such that $h^1\big(F,\cE(th)\big)=0$ for $t\in\bZ$.

\end{abstract}

\maketitle

\section{Introduction and Notation}
Throughout the whole paper we will work over an algebraically closed field $k$ of characteristic $0$ and we will denote by $\p 3$ the projective space of dimension $3$ over $k$. 

If $F\subseteq\p 3$ is a smooth surface we set $\cO_F(h):=\cO_{\p 3}(1)\otimes\cO_F$. A vector bundle $\cE$ on $F$ is called {\sl aCM} (i.e. arithmetically Cohen--Macaulay) if $h^1\big(F,\cE(th)\big)=0$ for $t\in\bZ$. 

When $\deg(F)\le2$, the well--known theorems of Horrocks and Kn\"orrer (see \cite{O--S--S}, \cite{A--O} and the references therein) provide a complete classification of aCM bundles on $F$. As $\deg(F)$ increases a complete picture is missing. It is thus interesting to better understand aCM bundles on surfaces with $\deg(F)\ge3$. 

In order to simplify our study we restrict our attention to bundles satisfying some further technical non--restrictive conditions. First, notice that the property of being aCM is invariant up to tensorizing by $\cO_F(th)$, thus we can consider in our study only {\sl initialized} bundles, i.e. bundles $\cE$ such that $h^0\big(F,\cE(-h)\big)=0$ and $h^0\big(F,\cE\big)\ne0$. Moreover, we can obviously focus onto {\sl indecomposable} bundles, i.e. bundles which do not split in a direct sum of bundles of lower rank.

When $\deg(F)=3$ some scattered results are known. E.g., M. Casanellas and R. Hartshorne proved that each smooth cubic surface supports families of arbitrary dimension of non--isomorphic, indecomposable initialized aCM bundles (see \cite{C--H1} and \cite{C--H2}). Moreover, D. Faenzi classified completely aCM bundles of rank $1$ and $2$ on such an $F$ in \cite{Fa}. A key point of all the quoted papers is the complete knowledge of $\Pic(F)$.

The case $\deg(F)\ge4$ is wide open. In this case the structure of $\Pic(F)$ is not uniquely determined, thus it could be very difficult to classify initialized indecomposable aCM bundles $\cE$ also in low ranks, unless assuming true some extra hypothesis on $F$. E.g. if $\Pic(F)$ is generated by $\cO_F(h)$ and $\cE$ has rank $2$ on $F$ with $c_1(\cE)\cong \cO_F(ch)$, then $3-\deg(F)\le c\le \deg(F)-1$, as essentially proved by C. Madonna (see \cite{Ma1}: see also \cite{Cs} where some slight improvement is explained). 
 
We are mostly interested in the case $\deg(F)=4$. In this case $F$ is a $K3$ surface, thus a lage amount of informations is available on $\Pic(F)$. E.g. K. Watanabe in \cite{Wa} classified initialized aCM bundles of rank $1$ on such an $F$. 

In the beautiful paper \cite{C--K--M1} E. Coskun, R.S. Kulkarni, Y. Mustopa dealt with the case of rank $2$ initialized indecomposable bundles $\cE$ on $F$ with $c_1=3h$. In \cite{C--N} the authors gave a complete description of bundles with $c_1=2h$: it is perhaps interesting to notice that in this case there is a large family of quartic surfaces, the {\sl determinantal} ones (i.e. surfaces whose equation can be expressed as the determinant of a $4\times4$ matrix with linear entries), supporting bundles whose non--zero sections do not vanish in the expected dimension.

The present paper is motivated by our interest in the classification of rank $2$ initialized indecomposable bundles $\cE$ on a determinantal quartic surface $F\subseteq\p3$. In this case $F$ contains a projectively normal smooth sextic curve $A$ of genus $3$. It is thus obvious that $\Pic(F)$ must contain the rank $2$ lattice $\Lambda$ generated by $h$ and $A$. 

Moreover, we actually have $\Pic(F) =\Lambda$ if $F$ is general in the sense of N\"other--Lefschetz, i.e. it represents a point of $\vert\cO_{\p3}(4)\vert$ outside a countable union of suitable proper subvarieties, due to some classic results proved in \cite{Lop}. 

Each $F$ satisfying the above property will be called a {\sl general determinantal quartic surface}. Theorems \ref{tc_1NonEffective} and \ref{tc_1Effective} give the complete classification of the bundles we are interested in on such an $F$. For reader's benefit we summarize such a classification in the following simplified statement. For the complete statement see Theorems \ref{tc_1NonEffective} and \ref{tc_1Effective}.

\medbreak
\noindent
{\bf Main Theorem.}
{\it Let $\cE$ be an initialized indecomposable aCM bundle of rank $2$ on a general determinantal quartic surface $F\subseteq\p3$.

Then one of the following asseritions holds.
\begin{enumerate}[label=(\roman{*})]
\item $c_1\in\{\ -h,A-h,2h-A\ \}$: then $c_2=1$. The zero locus $E$ of a general $s\in H^0\big(F,\cE)$ is a point.
\item $c_1=0$: then $c_2=2$. The zero locus $E\subseteq\p3$ of a general $s\in H^0\big(F,\cE)$ is a $0$--dimensional subscheme contained in exactly one line.
\item $c_1=h$: then $c_2$ is $3$, $4$, or $5$. The zero locus $E$ of a general $s\in H^0\big(F,\cE)$ is a $0$--dimensional subscheme of degree $c_2$ of a linear subspace of $\p3$ of dimension $c_2-2$.
\item $c_1\in\{\ A,3h-A\ \}$: then $c_2$  is $3$, or $4$. In this case the zero locus $E$ of a general $s\in H^0\big(F,\cE)$ is a $0$--dimensional subscheme.
\item $c_1\in\{\ A+h,4A-h\ \}$: then $c_2=8$. In this case the zero locus $E$ of a general $s\in H^0\big(F,\cE)$ is a $0$--dimensional subscheme which is contained in exactly two linearly independent quadrics.
\item $c_1=2h$: then $c_2=8$. In this case one of the two following cases occur:
\begin{enumerate}[label=(\alph{*})]
\item the zero locus $E$ of a general $s\in H^0\big(F,\cE)$ is a $0$--dimensional subscheme which is the base locus of a net of quadrics;
\item the zero locus $E$ of a general $s\in H^0\big(F,\cE)$ is a $0$--dimensional subscheme lying on exactly a twisted cubic;
\item the zero locus $E$ of each non--zero $s\in H^0\big(F,\cE)$ is a divisor in either $\vert A\vert$, or $\vert 3h-A\vert$.
\end{enumerate}
\item $c_1=3h$: then $c_2=14$. In this case the zero locus $E$ of a general $s\in H^0\big(F,\cE)$ is a $0$--dimensional subscheme.
\end{enumerate}
Moreover, all the above cases actually occur on $F$.
}
\medbreak

We now explain briefly the structure of the paper. In Section \ref{sQuartic} we recall all the facts on a smooth quartic surface that we will need in the paper. In Section \ref{sDeterminantal} we specialize to the case of determinantal quartic surfaces. In Section \ref{sExample} we list a series of examples. In Section \ref{sZero} we deal with the zero--locus of the sections of initialized indecomposable aCM bundles of rank $2$ on $F$, proving that it is $0$--dimensional, but the two cases which we mentioned before. Finally, in Section \ref{sMain} we prove the above Main Theorem, showing that the list of examples in Section \ref{sExample}.

The author would like to thank F. Galluzzi for several helpful discussions and the referee, whose comments helped to considerably improve the exposition of the paper.

For all the notations and unproven results we refer to \cite{Ha2}.

\section{Preliminary results}
\label{sQuartic}
In this section we will recall some preliminary facts about vector bundles on a smooth quartic surface $F\subseteq\p3$, most of them collected from several papers (see \cite{C--K--M2}, \cite{SD}, \cite{Wa} and \cite{Hu}). 

By adjunction in $\p3$ we have $\omega_F\cong\cO_F$, i.e. each such an $F$ is a $K3$ surface, thus we can make use of all the results proved in \cite{SD}. The first important fact is that Serre duality for each locally free sheaf $\cF$ on $F$ becomes
$$
h^i\big(F,\cF\big)=h^{2-i}\big(F,\cF^\vee\big),\qquad i=0,1,2.
$$
Moreover Riemann--Roch theorem on $F$ is
\begin{equation}
\label{RRGeneral}
h^0\big(F,\cF\big)+h^{2}\big(F,\cF\big)=h^{1}\big(F,\cF\big)+2\rk(\cF)+\frac{c_1(\cF)^2}2-c_2(\cF).
\end{equation}
In particular, if $D$ is a divisor with $D^2\ge-2$, then either $D$ or $-D$ is necessarily effective.

If $D$ is an effective divisor on $F$, then $h^2\big(F,\cO_F(D)\big)=h^0\big(F,\cO_F(-D)\big)=0$. Moreover
$$
h^1\big(F,\cO_F(D)\big)=h^1\big(F,\cO_F(-D)\big)=h^0\big(D,\cO_D\big)-1
$$
(see \cite{SD}, Lemma 2.2). It follows that for each irreducible divisor $D$ on $F$, the dimension of $\vert D\vert$ and the arithmetic genus of $D$ satisfy
\begin{equation}
\label{h^1}
h^0\big(F,\cO_F(D)\big)=2+\frac {D^2}2,\qquad \deg(D)=Dh,\qquad p_a(D)=1+\frac {D^2}2,
\end{equation}
(see \cite{SD}, Paragraph 2.4). In particular, the unique irreducible and reduced fixed divisors $D$ on $F$ satisfy $D^2=-2$, thus they are smooth rational curves.

We summarize the other helpful results we will need in the following statement.

\begin{proposition}
\label{SD}
Let $F\subseteq\p3$ be a smooth quartic surface.

For each effective divisor $D$ on $F$ such that $\vert D\vert$ has no fixed components the following assertions hold. 
\begin{enumerate}[label=(\roman{*})]
\item $D^2\ge0$ and $\cO_F(D)$ is globally generated.
\item If $D^2>0$, then the general element of $\vert D\vert$ is irreducible and smooth: in this case $h^1\big(F,\cO_F(D)\big)=0$.
\item If $D^2=0$, then there is an irreducible curve $\overline{D}$ with $p_a(\overline{D})=1$ such that $\cO_F(D)\cong\cO_F(e\overline{D})$ where $e-1:=h^1\big(F,\cO_F(D)\big)$: in this case the general element of $\vert D\vert$ is smooth.
\end{enumerate}
\end{proposition}
\begin{proof}
See \cite{SD}, Proposition 2.6 and Corollary 3.2.
\end{proof}

Now we turn our attention to aCM line bundles on $F$ recalling a very helpful result from \cite{Wa}. A first obvious fact is that for each  divisor $D$ on $F$, then $\cO_F(D)$ is aCM if and only if the same is true for $\cO_F(-D)$.

The main result from \cite{Wa} is the following statement.

\begin{proposition}
\label{pWa}
Let $F\subseteq\p3$ be a smooth quartic surface.

For each effective divisor $D$ on $F$ with $D\ne0$, then $\cO_F(D)$ is aCM if and only if one of the following cases occurs.
\begin{enumerate}[label=(\roman{*})]
\item $D^2=-2$ and $1\le hD\le3$.
\item $D^2=0$ and $3\le hD\le4$.
\item $D^2=2$ and $hD=5$.
\item $D^2=4$, $hD=6$ and $h^0\big(F,\cO_F(D-h)\big)=h^0\big(F,\cO_F(2h-D)\big)=0$.
\end{enumerate}
\end{proposition}
\begin{proof}
See \cite{Wa}, Theorem 1.1.
\end{proof}

Notice that the general smooth quartic surface $F\subseteq\p3$ does not support initialized aCM line bundles besides $\cO_F$. Indeed, in this case, $\Pic(F)$ is generated by $\cO_F(h)$, due to N\"other--Lefschetz theorem. In particular, for each line bundle $\cO_F(D)$ on $F$ one has that both $Dh$ and $D^2$ are positive multiples of $4$.

Recall that a curve $D\subseteq\p3$ is called aCM if $h^1\big(\p3,\cI_{D\vert\p3}(th)\big)=0$ for each $t\in\bZ$ (see \cite{Mi}, Lemma 1.2.3). A curve $D\subseteq\p3$ is projectively normal if and only if it is aCM and smooth.
We recall the following helpful fact.

\begin{lemma}
\label{laCM}
Let $F\subseteq\p3$ be a smooth quartic surface.

For each effective divisor $D$ on $F$, then $\cO_F(D)$ is  aCM if and only if the curve $D$ is aCM in $\p3$.
\end{lemma}
\begin{proof}
We have the exact sequence
\begin{equation*}
\label{seqManyIdeals}
0\longrightarrow \cI_{F\vert \p3}\longrightarrow \cI_{D\vert \p3}\longrightarrow \cI_{D\vert F}\longrightarrow 0.
\end{equation*}
Since $\cI_{F\vert \p3}\cong\cO_{\p3}(-4)$ and $\cI_{D\vert F}\cong\cO_F(-D)$, it follows that $h^1\big(\p3,\cI_{D\vert \p3}(t)\big)=0$ if and only if $h^1\big(F, \cO_F(D-th)\big)=h^1\big(F, \cI_{D\vert F}(th)\big)=0$.
\end{proof}

We conclude this section by recalling some further facts facts on vector bundles on a smooth quartic surface $F\subseteq\p3$. Let $\cF$ be a vector bundle on $F$: we can consider it as a sheaf over $\p3$, thus $H^0_*(\cF):=\bigoplus_{t\in\bZ} H^0\big(F,\cF(th)\big)$ is also a module over the polynomial ring $S:=k[x_0,x_1,x_2,x_3]$. For the following fact see \cite{C--H1}, Theorem 3.1.

\begin{lemma}
\label{lGenerator}
Let $\cF$ be an initialized aCM bundle of rank $r$ on a smooth quartic surface $F\subseteq\p3$.

The minimal number of generators of $H^0_*(\cF)$ as a module over $S$ is at most $4r$.
\end{lemma}

\begin{definition}
Let $\cF$ be an initialized aCM bundle of rank $r$ on a smooth quartic surface $F\subseteq\p3$.

We say that $\cF$ is {\sl Ulrich} if it is initialized, aCM and $h^0\big(F,\cF\big)=4r$. 
\end{definition}

\begin{remark}
\label{rUlrich}
If $\cF$ is Ulrich, then it is globally generated by Lemma \ref{lGenerator}. More generally, Ulrich bundles on $F$ of rank $r$ are exactly the bundles $\cF$ on $F$ such that $H^0_*(\cF)$ has linear minimal free resolution over $S$ (see \cite{C--H1}, Proposition 3.7).

Sheafifying such a resolution we thus obtain an exact sequence
\begin{equation*}
\label{seqUlrich}
0\longrightarrow\cO_{\p3}(-1)^{\oplus 4r}\longrightarrow\cO_{\p3}^{\oplus 4r}\longrightarrow\cF\longrightarrow0.
\end{equation*}
\end{remark}

\section{The general smooth determinantal quartic surfaces}
\label{sDeterminantal}
In this section we will summarize some facts about the smooth determinantal quartic surfaces in $\p3$.
We will denote by $f$ the polynomial defining $F$.

\begin{proposition}
\label{pDeterminantal}
Let $F\subseteq\p3$ be a smooth quartic surface. 

The following assertions are equivalent.
\begin{enumerate}[label=(\roman{*})]
\item The polynomial $f$ defining $F$ is the determinant of a $4\times4$ matrix $\Phi$ with linear entries.
\item There exists a line bundle $\cO_F(A)$ such that the sheafified minimal free resolution of $H^0_*(\cO_F(A))$ has the form
\begin{equation}
\label{seqDeterminantal}
0\longrightarrow\cO_{\p3}(-1)^{\oplus4}\mapright{\varphi}\cO_{\p3}^{\oplus4}\longrightarrow\cO_F(A)\longrightarrow0.
\end{equation}
\item There exists an initialized aCM line bundle $\cO_F(A)$ with $A^2=4$.
\item There exists a smooth integral sextic curve $A\subseteq F$ of genus $3$ which is not hyperelliptic.
\item There exists a smooth integral sextic curve $A\subseteq F$ of genus $3$ which is projectively normal in $\p3$.
\item There exists a smooth integral sextic curve $A\subseteq F$ of genus $3$ such that
$$
h^0\big(F,\cO_F(2h-A)\big)=h^0\big(F,\cO_F(A-h)\big)=0.
$$
\end{enumerate}
\end{proposition}
\begin{proof}
Assertion (i) is equivalent to assertions (ii), (iv), (v) due to Corollaries 1.8, 6.6 and Proposition 6.2 of \cite{Bea}. Assertions (iii) and (vi) are equivalent thanks to Proposition \ref{pWa}.

Assertion (ii) implies (iii). Indeed, we have
$$
2+\frac{A^2}2=\chi(\cO_F(A))=4(\chi(\cO_{\p3})-\chi(\cO_{\p3}(-1)))=4,
$$
thus $A^2=4$. Conversely, $h^0\big(F,\cO_F(A)\big)=4$ for each initialized aCM line bundle with $A^2=4$. Hence, $\cO_F(A)$ is Ulrich, thus it fits into Sequence \eqref{seqDeterminantal} (see e.g. Proposition 3.7 of \cite{C--H1}).
\end{proof}

\begin{remark}
Let $F$ be determinantal, so that we have Sequence \eqref{seqDeterminantal}. Its dual is\begin{equation}
\label{seqDual}
0\longrightarrow\cO_{\p3}(-1)^{\oplus4}\mapright{{}^t\varphi}\cO_{\p3}^{\oplus4}\longrightarrow\cO_F(3h-A)\longrightarrow0,
\end{equation}
because $\Ext^1_{\p3}\big(\cO_F(A),\cO_{\p3}\big)\cong\cO_F(4h-A)$.

Thus also the general element in $\vert3h-A\vert$ is a curve enjoying the same properties as $A$. Notice that $\cO_F(3h-A)\not\cong\cO_F(A)$, because $(3h-A)A=14\ne4=A^2$.
\end{remark}

From the above results we deduce that if the polynomial $f$ defining the smooth quartic surface $F$ is the determinant of a $4\times4$ matrix with linear entries, then $\Pic(F)$ contains the even sublattice $\Lambda$ generated by $h$ and $A$ such that $h^2=A^2=4$ and $hA=6$. 

By definition the general element in $\vert h\vert$ is an irreducible smooth plane curve of degree $4$. Thanks to Proposition \ref{pDeterminantal} the general element in $\vert A\vert$ is an integral smooth projectively normal curve of degree $6$ and genus $3$(see Equalities \eqref{h^1}).

Conversely, let $A\subseteq\p3$ be an integral projectively normal sextic of genus $3$. We know that each smooth quartic surface $F\subseteq\p3$ containing $A$ is necessarily determinantal. The minimal degree of a surface containing such an $A$ is $3$, because $h^0\big(F,\cO_F(2h-A)\big)=0$. Moreover $\cO_F(4h-A)$ is globally generated (see e.g. Sequence \eqref{seqDual}), thus the general element in $\vert\cO_F(4h-A)\vert$ is integral. It follows that $\Pic(F)=\Lambda$ due to \cite{Lop}, Theorem II.3.1, if $F$ is general in the N\"other--Lefschetz sense. 

From now on, we will assume that the surface $F$ has the above property and we will call it a {\sl general determinantal quartic surface}\/. 

When we will consider such a surface $F$, we will always denote by $\cO_F(A)$ a line bundle corresponding to a projectively normal sextic of genus $3$ on $F$. Notice that there is an ambiguity because, as we pointed out above, the line bundle $\cO_F(3h-A)$ has the same properties as $\cO_F(A)$.

In order to complete the picture, we recall that not all determinantal quartic surfaces $F\subseteq\p3$ satisfy $\rk(\Pic(F))=2$.

E.g. a smooth quartic surface $F\subseteq\p3$ containing a twisted cubic curve $T$ and a line $E$ without common points is determinantal (see \cite{Do}, Exercise 4.14). Indeed the union of $T$ with any plane cubic linked to $E$ define a linear system whose general element $A$ is a projectively normal sextic of genus $3$. It is easy to check that $h$, $T$, $E$ are linearly independent in $\Pic(F)$, thus $\rk(\Pic(F))\ge 3$ in this case.

In \cite{C--G} the authors construct many examples of determinantal quartic surfaces $F\subseteq\p3$ with $\rk(\Pic(F))=3$ and show that even the classification of indecomposable Ulrich bundles of rank $2$ on them seems to be highly non--trivial.

\begin{remark}
\label{rIntersection}
A first interesting fact is that if $D\in \vert xh+yA\vert$, then 
\begin{equation}
\label{Intersection}
D^2=4x^2+12xy+4y^2,\qquad Dh=4x+6y,\qquad DA=6x+4y.
\end{equation}
More precisely it is very easy to check the following assertions: 
\begin{enumerate}
\item[(i)] $D^2\equiv0\pmod{16}$ if and only if both $x$ and $y$ are even;
\item[(ii)] $D^2\equiv4\pmod8$ if and only if $x$ and $y$ are not simultaneously even.
\end{enumerate}
The two above assertions imply that $D^2$ is always a multiple of $4$. Thus Equality \eqref{RRGeneral} imply that $\chi(\cO_F(D))$ is even. Moreover, Equalities \eqref{Intersection} imply that the equation $D^2=0$ has no non--trivial integral solutions, thus:
\begin{enumerate}
\item[(iii)] $D^2=0$ if and only if $D=0$.
\end{enumerate}
Irreducible and reduced fixed divisors $D$ on $F$ satisfy $D^2=-2$ (see the last Equality \eqref{h^1}), which is not a multiple of $4$. It follows that there are no fixed divisors on $F$. 
In particular $DU\ge0$ for each divisor $U\subseteq F$, hence (see Equalities \eqref{Intersection}):
\begin{enumerate}
\item[(iv)] $DU\ge 0$ is even for each divisor $U\subseteq F$.
\end{enumerate}
Combining the above assertions with Proposition \ref{SD}, we deduce the following consequence:
\begin{enumerate}
\item[(v)] each element in $\vert h\vert$, $\vert A\vert$, $\vert 3h-A\vert$ is certainly integral (but maybe singular).
\end{enumerate}
\end{remark}

Looking at Remark \ref{rIntersection} we know that the following result holds.

\begin{lemma}
\label{lFixed}
Let $F\subseteq\p3$ be a general determinantal quartic surface. 

If the divisor $D$ on $F$ is effective, then $\vert D\vert$ has no fixed components. 
\end{lemma}

In particular for each effective divisor $D$ on $F$ we have $D^2\ge0$ and  $h^0\big(F,\cO_F(D)\big)\ge2$. Moreover, if $D\ne0$, then from Remark \ref{rIntersection} we know that $D^2\ne0$, whence $D^2\ge4$ and  $h^0\big(F,\cO_F(D)\big)\ge4$. Finally, the ampleness of $h$ implies $Dh\ge2$.

The above discussion partially proves the following result.

\begin{lemma}
\label{lD^2ge4}
Let $F\subseteq\p3$ be a general determinantal quartic surface. 

The divisor $D$ on $F$ is effective if and only if either $D=0$, or $Dh\ge2$ and $D^2\ge4$. 
\end{lemma}
\begin{proof}
The \lq only if\rq\ part was proved above. Assume that $Dh\ge2$ and $D^2\ge4$. Either $D$ or $-D$ is effective, because $D^2\ge-2$. We have $(-D)h=-2$, thus $-D$ is not effective. we conclude that $D$ is effective.
\end{proof}

\begin{remark}
\label{rEffective}
Thus $D\in \vert xh+yA\vert$ is effective if and only if either $(x,y)=(0,0)$, or
$$
x^2+3xy+y^2\ge 1,\qquad 2x+3y\ge1.
$$
The above inequalities give the closed connected subset of the affine plane with coordinates $x$ and $y$ not containing the point $(x,y)=(0,0)$ and whose border is the graph of the function
$$
y=\frac{-3x+\sqrt{5x^2+4}}2
$$
(the \lq upper branch\rq\ of the hyperbola $x^2+3xy+y^2=1$).
\end{remark}

Proposition \ref{SD}  and the above Lemmas \ref{lFixed} and \ref{lD^2ge4} prove the following corollary.

\begin{corollary}
\label{cGG}
Let $F\subseteq\p3$ be a general determinantal quartic surface. 

The divisor $D$ on $F$ is effective if and only if $\cO_F(D)$ is globally generated.
\end{corollary}

\section{Examples of rank $2$ aCM bundles}
\label{sExample}
In this section we briefly recall Hartshorne--Serre correspondence, listing some examples of rank $2$ aCM bundles on general determinantal quartic surfaces. 

Let $\cE$ be a rank $2$ vector bundle on a smooth quartic surface $F\subseteq\p3$ with Chern classes $c_i:=c_i(\cE)$, $i=1,2$. If $s\in H^0\big(F,\cE\big)$ is non--zero, then its zero--locus $E:=(s)_0\subseteq F$ has codimension at most $2$. 

We can write $(s)_0=E\cup D$ where $E$ has pure codimension $2$ (or it is empty) and $D$ is a possibly zero divisor. We thus obtain a section $\sigma\in H^0\big(F,\cE(-D)\big)$ vanishing exactly on $E$. The Koszul complex of $\sigma$ gives the following exact sequence 
\begin{equation}
\label{seqIdeal}
0\longrightarrow \cO_F(D)\longrightarrow\cE\longrightarrow \cI_{E\vert F}(c_1-D)\longrightarrow 0.
\end{equation}
The degree of $E$ is $c_2(\cE(-D))$. 

Bertini's theorem for the sections of a vector bundle guarantees that if $\cE$ is globally generated (e.g. if $\cE$ is Ulrich), then $D=0$ for the general $s\in H^0\big(F,\cE\big)$.

Assume that the divisorial part $D$ of $(s)_0$ is not empty for general $s\in H^0\big(F,\cE\big)$. We can define a rational map $H^0\big(F,\cE)\dashrightarrow\Pic(F)$ defined by $s\mapsto\cO_F(D)$. Since $\Pic(F)$ is discrete, it follows that such a map is constant. Thus the divisorial part of the zero--loci of the sections in a suitable non--empty subset of $H^0\big(F,\cE)$ are linearly equivalent. 

The  above construction can be reversed. To this purpose we give the following well--known definition.

\begin{definition}
Let $F\subseteq\p3$ be a smooth quartic surface and let $\cG$ be a coherent sheaf on $F$.

We say that a locally complete intersection subscheme $E\subseteq F$ of dimension zero is Cayley--Bacharach (CB for short) with respect to $\cG$ if, for each $E'\subseteq E$ of degree $\deg(E)-1$, the natural morphism $H^0\big(F,\cI_{E\vert F}\otimes\cG\big)\to H^0\big(F,\cI_{E'\vert F}\otimes\cG\big)$ is an isomorphism.
\end{definition}

For the following result see Theorem 5.1.1 in \cite{H--L}.

\begin{theorem}
\label{tCB}
Let $F\subseteq\p3$ be a smooth quartic surface, $E\subseteq F$ a locally complete intersection subscheme of dimension $0$. 
 
Then there exists a vector bundle $\cE$ of rank $2$ on $F$ with $\det(\cE)={\mathcal L}$ and having a section $s$ such that $E=(s)_0$ if and only if $E$ is CB with respect to ${\mathcal L}$. 

If such an $\cE$ exists, then it is unique up to isomorphism if $h^1\big(F,{\mathcal L}^\vee\big)= 0$.
\end{theorem}

Let $F\subseteq\p3$ be a smooth quartic surface. We recall that in \cite{C--K--M1} the author proved the existence on $F$ of aCM bundles $\cE$ with $c_1(\cE)=3h$ and $c_2(\cE)=14$. To this purpose the authors proved the existence of suitable sets of $14$ points which are CB with respect to $\cO_F(3h)$.

Similarly, in \cite{C--N} the author proved the existence on $F$ of aCM bundles $\cE$ with $c_1(\cE)=2h$ and $c_2(\cE)=8$. Essentially two cases are possible. 

In the first case, the zero--locus of a general section of $\cE$ has pure dimension $0$: this case occurs without further hypothesis on $F$. 

In the second case, the zero--loci of non--zero sections of $\cE$ define a linear system $\vert D\vert$ whose general element is a smooth integral sextic of genus $3$ such that $h^0\big(F,\cO_F(D-h)\big)=h^0\big(F,\cO_F(2h-D)\big)=0$. In particular $F$ must be a determinantal quartic surface in this second case.

Due to the results listed in the previous section we know that this case can occur only if $F$ is determinantal. In particular it occurs when $F$ is a general determinantal quartic surface $F$.

We now construct examples of rank $2$ aCM bundles starting from suitable sets of points $E\subseteq F$. Recall that for each closed subscheme $X\subseteq F$ we have an exact sequence
\begin{equation}
\label{seqStandard}
0\longrightarrow\cI_{X\vert F}\longrightarrow\cO_F\longrightarrow\cO_X\longrightarrow0.
\end{equation}

In what follows we will present a series of examples of initialized indecomposable aCM rank $2$ bundles on a quartic surface $F\subseteq\p3$ with plane section $h$. We stress that the bundles defined in Examples \ref{c_1=1} and \ref{c_1=0,-1} exist on each smooth quartic surface. 

\begin{example}
\label{c_1=1}
Let $E\subseteq F$ be a $0$--dimensional subscheme arising as follows:
\begin{itemize}
\item{$m=3$} $E$ is the complete intersection of a line and a cubic surface;
\item{$m=4$} $E$ is the complete intersection of two conics inside a plane;
\item{$m=5$} $E$ is the degeneracy locus of a $5\times5$ skew symmetric matrix $\Delta$ with linear entries.
\end{itemize}
It is clear that $E$ is CB with respect to $\cO_F(h)$ in the first two cases. In the third one, a minimal free resolution of $H^0_*(\cO_E)$  over $\p3$ is given by the Buchsbaum--Eisenbud complex (see \cite{B--E}). Thus again $E$ is CB with respect to $\cO_F(h)$ thanks to  Lemma (4.2) of \cite{Sch}.

Theorem \ref{tCB} thus yields the existence of a rank $2$ bundle $\cE$ fitting into Sequence \eqref{seqIdeal} with $c_1=h$, $c_2=m$ and $D=0$, i.e.
\begin{equation}
\label{seqc_1=1}
0\longrightarrow \cO_F\longrightarrow \cE\longrightarrow \cI_{E\vert F}(h)\longrightarrow 0,
\end{equation}
The cohomology of the above sequence twisted by $\cO_F(-h)$ gives
$$
h^2\big(F,\cE\big)=h^0\big(F,\cE^\vee\big)=h^0\big(F,\cE(-h)\big)=0,
$$
i.e. $\cE$ is initialized and the cohomology of Sequence \eqref{seqc_1=1} yields the exact sequence. 
$$
0\longrightarrow H^1\big(F,\cE\big)\longrightarrow H^1\big(F,\cI_{E\vert F}(h)\big)\longrightarrow H^2\big(F,\cO_F\big)\longrightarrow 0.
$$
By definition $h^0\big(F,\cI_{E\vert F}(h)\big)=5-m$, whence $h^1\big(F,\cI_{E\vert F}(h)\big)=1$, from the cohomology of Sequence \eqref{seqStandard}. We conclude that $h^1\big(F,\cE\big)=0$.

The cohomology of Sequence \eqref{seqc_1=1} twisted by $\cO_F(th)$ gives 
$$
h^1\big(F,\cE(th)\big)=h^1\big(F,\cI_{E\vert F}((t+1)h)\big),\qquad t\ge1.
$$ 

In the cases $m=3,4$ the minimal free resolution of $H^0_*(\cO_E)$ over $\p3$ is given by a Koszul complex. In the case $m=5$, it is given by the Buchsbaum--Eisenbud complex as already pointed out above. In any case it is easy to check that the Hilbert function of $E$ is $(1,m-1,m,\dots)$, hence the natural map $H^0\big(F,\cO_F(th)\big)\to H^0\big(E,\cO_E(th)\big)$ is surjective for $t\ge1$, thus $h^1\big(F,\cI_{E\vert F}(th)\big)=0$, in the same range (see the cohomology of Sequence \eqref{seqStandard}).
We conclude that $h^1\big(F,\cE(th)\big)=0$ for $t\ge0$.
By duality
$$
h^1\big(F,\cE(th)\big)=h^1\big(F,\cE^\vee(-th)\big)=h^1\big(F,\cE((-t-1)h)\big)=0,\qquad -t-1\ge0,
$$
i.e. $t\le-1$.
We conclude that $\cE$ is aCM. 

If $\cE\cong\cO_F(U)\oplus\cO_F(V)$, then $U+V=h$ and $c_2=UV$ which is even (see Remark \ref{rIntersection}). Moreover both $\cO_F(U)$ and $\cO_F(V)$ are aCM and at least one of them must be initialized, the other being either initialized or without sections. 

Let $\cO_F(U)$ be initialized. If $U=A$, then $V=h-A$, thus $UV=2<3\le c_2$. The case $U=3h-A$ can be excluded in the same way. If $U=0$, then $V=h$, a contradiction. We conclude that $\cE$ is indecomposable.
\end{example}

\begin{example}
\label{c_1=0,-1}
Let $E\subseteq F$ be a set of $m=1,2$ points. Then $E$ is trivially CB with respect to $\cO_F((m-2)h)$, thus there is a rank $2$ bundle $\cE$ fitting into Sequence \eqref{seqIdeal} with $c_1=(m-2)h$, $c_2=m$ and $D=0$. 

As in the Example \ref{c_1=1} one easily checks that such an $\cE$ is aCM, initialized and indecomposable.
\end{example}

For all the following examples we obviously need to restrict to a determinantal quartic surface $F\subseteq\p3$. As usual $h$ denotes the class of the plane section and $A$ the class of a projectively normal integral sextic curve  of genus $3$ on $F$.

\begin{example}
\label{c_1=A-h}
Let $E\subseteq F$ be a single point. The line bundle $\cO_F(A-h)$ is not effective, thus $E$ is again trivially CB with respect to $\cO_F(A-h)$. In particular there is a rank $2$ bundle $\cE$ fitting into an exact sequence of the form
\begin{equation}
\label{seqc_1=A-h}
0\longrightarrow \cO_F\longrightarrow \cE\longrightarrow \cI_{E\vert F}(A-h)\longrightarrow 0,
\end{equation}
The cohomology of Sequence \eqref{seqc_1=A-h} twisted by $\cO_F(-h)$ gives $h^0\big(F,\cE(-h)\big)=0$, because $\cO_F(A-2h)$ is not effective. Thus $\cE$ is initialized. The same sequence twisted by $\cO_F(th)$ also yields the existence of an injective map $H^1\big(F,\cE(th)\big)\to H^1\big(F,\cI_{E\vert F}(A+(t-1)h)\big)$ for each $t$.

The cohomology of Sequence \eqref{seqStandard} twisted by $\cO_F(A+(t-1)h)$ yields $h^1\big(F,\cI_{E\vert F}(A+(t-1)h)\big)=0$ for $t\ge1$. Hence $h^1\big(F,\cE(th)\big)=0$, for $t\ge1$.

The cohomology of Sequence \eqref{seqStandard} twisted by $\cO_F(-th)$ yields $h^1\big(F,\cI_{E\vert F}(-th)\big)=0$ for $t\le0$. Sequence \eqref{seqc_1=A-h} yields the existene of surjective maps
$$
H^1\big(F,\cO_F((1-t)h-A)\big)\to H^1\big(F,\cE^\vee(-th)\big),\qquad t\le 0.
$$
Since $\cO_F(A)$ is aCM, the same is true for $\cO_F(-A)$, thus $h^1\big(F,\cE(th)\big)=h^1\big(F,\cE^\vee(-th)\big)=0$ for $t\le 0$.
We conclude that $\cE$ is aCM. 

If $\cE\cong\cO_F(U)\oplus\cO_F(V)$, then $c_2=UV$ which is even (see Remark \ref{rIntersection}). Since $c_2=1$, it follows that $\cE$ is indecomposable. Finally $c_1=A-h$ by construction.

Substituting $A$ with $3h-A$ in the above construction we are also able to construct an example of initialized indecomposable aCM bundle of rank $2$ with $c_1=2h-A$.

We can construct two more bundles. Indeed, $\cE^\vee(h)$ is indecomposable and aCM too. Moreover the cohomology of Sequence \eqref{seqc_1=A-h} suitably twisted also implies that $\cE^\vee(h)$ is initialized. Computing the Chern classes of $\cE^\vee(h)$ we obtain $c_1=3h-A$ and $c_2=3$. Again the substitution of $A$ with $3h-A$ gives an example of initialized indecomposable aCM bundle of rank $2$ with $c_1=A$ and $c_2=3$.
\end{example}

\begin{example}
\label{c_1=A}
The curve $A\subseteq F$ has genus $3$ and it is not hyperelliptic (see Proposition \ref{pDeterminantal}), thus it supports infinitely many complete $g^1_4$. Fix a smooth divisor $E$ in one of these $g^1_4$. 
The natural composite map $\cO_F^{\oplus2}\to\cO_A^{\oplus2}\to\cO_A(E)$ is surjective on global sections and its kernel is a vector bundle $\cF$ on $F$. The bundle $\cE:=\cF^\vee$ is  called the {\sl Lazarsfeld--Mukai bundle of $\cO_A(E)$}. By definition we have the exact sequence
\begin{equation}
\label{seqMukai}
0\longrightarrow \cE^\vee\longrightarrow\cO_F^{\oplus2}\longrightarrow \cO_A(E)\longrightarrow 0.
\end{equation}
Clearly $c_1=A$ and $c_2=4$. In what follows we will quickly show that $\cE$ is aCM.

The cohomology Sequence \eqref{seqMukai} yields $h^1\big(F,\cE\big)=h^1\big(F,\cE^\vee\big)=0$ by definition. Twisting Sequence \eqref{seqMukai} by $\cO_F(-th)$, taking its cohomology we also obtain $h^1\big(F,\cE(-th)\big)=0$ for $t\ge1$, because $h^0\big(A,\cO_A(E-th)\big)=0$, due to the inequality $\deg(\cO_A(E-th))=4-6t\le -2$ in the same range.

We have $\Ext^1_F\big(\cO_A(E),\cO_F\big)\cong\omega_A(-E)$. Adjunction formula on $F$ thus implies $\omega_A\cong\cO_A(A)$, hence the dual of Sequence \eqref{seqMukai} is 
\begin{equation}
\label{seqMukaiDual}
0\longrightarrow \cO_F^{\oplus2}\longrightarrow\cE\longrightarrow \cO_A(A-E)\longrightarrow 0.
\end{equation}
Thus, also for $t\ge1$, we obtain
$$
h^1\big(F,\cE(th)\big)\le h^1\big(A,\cO_{A}(A-E+th)\big)=h^0\big(A,\cO_A(E-th)\big)=0.
$$
We conclude that $\cE$ is aCM. The cohomology of Sequence \eqref{seqMukaiDual} twisted by $\cO_F(-h)$ and the vanishing $h^0\big(A,\cO_A(E-h)\big)=0$ (see above) also imply that $\cE$ is initialized.

Assume that  $\cE\cong\cO_F(U)\oplus\cO_F(V)$. We thus have  $U+V=A$ and $UV=4$. As in Example \ref{c_1=1} we can assume that $\cO_F(U)$ is aCM and initialized and that $\cO_F(V)$ is either initialized or without sections. If $U=A$, then $V=0$, whence $UV=0$ a contradiction. Similarly $U\ne0$. If $U=3h-A$, then $V=2A-3h$ and again $UV=0$.

The same argument with $3h-A$ instead of $A$ yields the existence of initialized indecomposable aCM bundles of rank $2$ with $c_1=3h-A$ and $c_2=4$.

Also in this case we can construct two more bundles. Indeed $\cE^\vee(h)$ is indecomposable and aCM too. Moreover, Equality \eqref{RRGeneral} and the vanishing $h^2\big(F,\cE^\vee(th)\big)=h^0\big(F,\cE(-th)\big)=0$, $t=1,2$, return
$$
h^0\big(F,\cE^\vee(h)\big)=0,\qquad
h^0\big(F,\cE^\vee(2h)\big)=6.
$$
It follows that  $\cE^\vee(2h)$ is initialized. For such a bundle $c_1=4h-A$ and $c_2=8$. The substitution of $A$ with $3h-A$ gives an example of initialized indecomposable aCM bundle of rank $2$ with $c_1=A+h$ and $c_2=8$.

It is interesting to notice that the bundle $\cE$ (hence also $\cE^\vee(2h)$) described above has an interesting property. Indeed, one can prove using Lemma 2.4 of \cite{Knu2} that both $\cO_F(A)$, and $\cO_F(3h-A)$ are very ample on $F$, thus they define two other different embeddings $F\subseteq\p3$ with respective hyperplane line bundles $\cO_F(A)$ and $\cO_F(3h-A)$. Proposition 2.3 of \cite{Wa2} guarantees that $\cE$ is aCM also for these embeddings: indeed, it is the bundle with $m=4$ of Example \ref{c_1=1} with respect to these embeddings.
\end{example}

\section{On the zero--locus of general sections}
\label{sZero}

In this section we will assume that $F\subseteq\p3$ is a general determinantal quartic surface. 

\begin{lemma}
\label{lNonEmpty}
Let $\cE$ be an indecomposable aCM bundle of rank $2$ on a general determinantal quartic surface $F\subseteq\p3$.

If $s\in H^0\big(F,\cE\big)$, then $(s)_0\ne\emptyset$.
\end{lemma}
\begin{proof}
Assume that $(s)_0=\emptyset$. Then Sequence \eqref{seqIdeal} becomes
$$
0\longrightarrow \cO_F\longrightarrow\cE\longrightarrow\cO_F(c_1)\longrightarrow0,
$$
where $h^1\big(F,\cO_F(c_1)\big)=h^1\big(F,\cO_F(-c_1)\big)\ne0$, because $\cE$ is indecomposable. The cohomology of the above sequence gives a monomorphism $H^1\big(F,\cO_F(c_1)\big)\to H^2\big(F,\cO_F\big)$, thus $h^1\big(F,\cO_F(c_1)\big)=1$.

Remark \ref{rIntersection} and Equality \eqref{RRGeneral} imply that $\chi(\cO_F(c_1))$ is even, thus either $\cO_F(c_1)$ or $\cO_F(-c_1)$ is effective. Hence $c_1^2\ge4$ (see Lemma \ref{lD^2ge4}), thus $h^1\big(F,\cO_F(c_1)\big)=0$ (see Lemma \ref{lFixed} and Proposition \ref{SD}), a contradiction.
\end{proof}

Trivially $h^1\big(F,\cE^\vee(2h)\big)=0$: moreover
$h^2\big(F,\cE^\vee(h)\big)=h^0\big(F,\cE(-h)\big)=0$.
The above vanishings prove the following lemma.

\begin{lemma}
\label{lGG}
Let $\cE$ be an initialized, indecomposable aCM bundle of rank $2$ on a general determinantal quartic surface $F\subseteq\p3$.

Then $\cE^\vee(3h)$ is regular in the sense of Castelnuovo--Mumford. In particular, both $\cE^\vee(3h)$ and $\cO_F(6h-c_1)$ are globally generated.
\end{lemma}

Let $\cE$ be as above.
If $s\in H^0\big(F,\cE\big)$ is a non--zero section, then $(s)_0\ne\emptyset$. From now on we will assume that $(s)_0=E\cup D$ where $E$ has pure dimension $0$ (or it is empty) and $D$ is an effective divisor. 

Notice that the cohomology of Sequence \eqref{seqIdeal} gives an injective map $H^0\big(F,\cO_F(D-h)\big)\to H^0\big(F,\cE(-h)\big)$, thus $\cO_F(D)$ is initialized too, or, in other words, $\cO_F(D-h)$ is not effective.

Now assume that $c_1=xh+yA$ is effective. It follows from Lemma \ref{lGG} and Remark \ref{rEffective} that $c_1$ is either $0$, or $6h$, or the inequalities
\begin{gather*}
y^2+3xy+x^2\ge1,\qquad 2x+3y\ge1,\\
y^2-3(6-x)y+(6-x)^2\ge1,\qquad 2(6-x)-3y\ge1,
\end{gather*}
must hold. Thus, in the latter case, the point $(x,y)$ must lie in the region delimited by the graphs of the two functions
$$
y=\frac{-3x+\sqrt{5x^2+4}}{2},\qquad y=\frac{3(6-x)-\sqrt{5(6-x)^2+4}}{2}.
$$
By drawing the graphs of the above functions (we made such computations using the software Grapher), we obtain Picture 1 below.

\vglue-0.5truecm

\begin{figure}[h!!]
\centering
\includegraphics[draft=false,width=8cm,height=5.33cm]{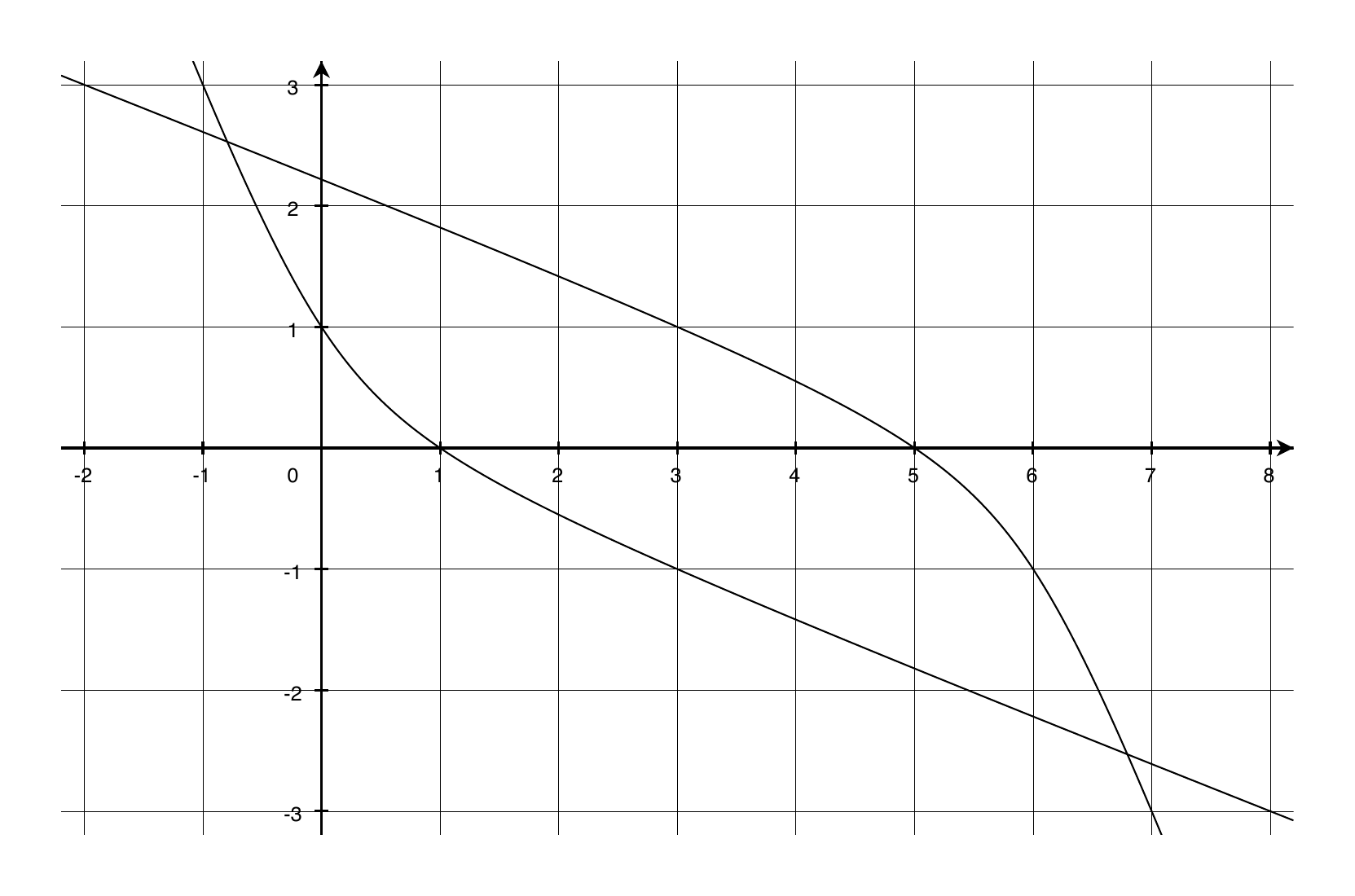}
\vglue-0.5truecm
\caption{}
\end{figure}

\noindent The values of $c_1$ we are interested in are in the region bounded by the graphs of the two functions above,  plus the two points $(0,0)$ and $(6,0)$ (see above). Thus $c_1$ is in the following list:
\begin{equation}
\label{listc_1}
\begin{gathered}
0,\quad A,\quad 2A,\quad h,\quad h+A,\quad 2h,\quad 2h+A,\quad 3h-A,\quad 3h,\quad 3h+A,\\
 4h-A,\quad 4h,\quad 5h-A,\quad 5h,\quad 6h-2A,\quad 6h-A,\quad 6h.
\end{gathered}
\end{equation}

We are ready to prove the main result of the present section.

\begin{theorem}
\label{tc_1Classification}
Let $\cE$ be an initialized indecomposable aCM bundle of rank $2$ on a general determinantal quartic surface $F\subseteq\p3$.

Let $s\in H^0\big(F,\cE\big)$ and $(s)_0=E\cup D$ be as above. Then exactly one of the following cases occurs:
\begin{enumerate}[label=(\roman{*})]
\item $c_1-D$ is effective;
\item $E$ is a point, $D=0$ and $c_1$ is not effective;
\item $E=\emptyset$, $D\in \vert A\vert \cup\vert 3h-A\vert$ and $c_1=2h$.
\end{enumerate}
\end{theorem}
\begin{proof}
Let $c_1-D$ be effective. It is obviously true that $c_1$ is effective, thus case (ii) above cannot occur.

If $c_1=2h$ and $D\in \vert A\vert \cup\vert 3h-A\vert$, then $c_1-D$ is not effective, because $(2h-A)^2=(A-h)^2=-4$, (see Lemma \ref{lD^2ge4}). It turns out that case (iii) above cannot occur as well.

In order to complete the proof we can thus assume that $c_1-D$ is not effective proving that case either (ii), or (iii) must hold. The cohomology of Sequence \eqref{seqStandard} for $X:=E$ twisted by $\cO_F(c_1-D)$ implies
$$
h^1\big(F,\cI_{E\vert F}(c_1-D)\big){\ge} \deg(E).
$$
The cohomology of Sequence \eqref{seqIdeal} implies
$$
h^1\big(F,\cI_{E\vert F}(c_1-D)\big)\le h^2\big(F,\cO_F(D)\big)=h^0\big(F,\cO_F(-D)\big).
$$

If $D=0$, then $c_1$ is not effective. We have $E\ne\emptyset$ due to Lemma \ref{lNonEmpty}. Moreover 
$$
1\le \deg(E)\le h^1\big(F,\cI_{E\vert F}(c_1)\big)\le h^0\big(F,\cO_F\big)=1,
$$
thus equality must hold. It follows that $E$ is a point.

If $D\ne0$, then 
$$
\deg(E)\le h^1\big(F,\cI_{E\vert F}(c_1-D)\big)\le h^0\big(F,\cO_F(-D)\big)=0,
$$
because $D$ is effective and non--zero. Thus $E=\emptyset$ and Sequence \eqref{seqIdeal} becomes
\begin{equation*}
\label{seqEEmpty}
0\longrightarrow \cO_F(D)\longrightarrow\cE\longrightarrow\cO_F(c_1-D)\longrightarrow0.
\end{equation*}
The cohomology of the above sequence gives surjective maps $H^0\big(F,\cO_F(c_1-D+th)\big)\to H^1\big(F,\cO_F(D+th)\big)$, for $t\in\bZ$. It follows that $h^1\big(F,\cO_F(D+th)\big)=0$ for each $t\le 0$. Moreover $D$ is effective, thus $h^1\big(F,\cO_F(D+th)\big)=0$ also for $t\ge1$. We conclude that $D$ is aCM.

We already know that $\cO_F(D)$ is initialized, because the same is true for $\cE$. Due to Proposition \ref{pWa} and Lemma \ref{lD^2ge4} the only initialized aCM line bundles $\cO_F(xh+yA)\not\cong\cO_F$ on $F$ satisfy the equalities $(xh+yA)^2=4$ and $(xh+yA)h=6$, which are equivalent to the system
$$
x^2+3xy+y^2-1=2x+3y-3=0.
$$
The unique solutions of this system correspond exactly to $A$ and $3h-A$.

If $c_1$ is effective, then it is in the List \eqref{listc_1}. Taking into account that $c_1-D$ is not effective, it follows that $D$ is non--zero, whence $E=\emptyset$ from the above discussion. The pair $(c_1,D)$ must be either $(2h,A)$, or $(2h,3h-A)$, or it is in the following list;
\begin{equation*}
\label{listc_1D}
\begin{gathered}
(0,A),\quad (0,3h-A),\quad (A,3h-A),\quad (2A,3h-A),\\
(h+A,3h-A),\quad (2h+A,3h-A),\quad (h,A),\quad (h,3h-A),\\
 (3h-A,A),\quad
(4h-A,A),\quad(5h-A,A),\quad(6h-2A,A).
\end{gathered}
\end{equation*}
In order to complete the proof of the theorem we have to show that all the cases in the above list cannot occur. We will examine the cases one by one.

Let $(c_1,D)=(0,A)$. Sequence \eqref{seqIdeal} becomes
$$
0\longrightarrow \cO_F(A)\longrightarrow \cE\longrightarrow \cO_F(-A)\longrightarrow 0,
$$
which splits, because $h^1\big(F,\cO_F(2A)\big)=0$, contradicting the indecomposability of $\cE$. The same argument can be used to exclude also the case $(0,3h-A)$.

Let $(c_1,D)=(h,A)$. Sequence \eqref{seqIdeal} becomes
$$
0\longrightarrow \cO_F(A)\longrightarrow \cE\longrightarrow \cO_F(h-A)\longrightarrow 0.
$$
Since $(2A-h)^2=-4$, we deduce that $h^0\big(F,\cO_F(2A-h)\big)=h^1\big(F,\cO_F(2A-h)\big)=0$. Equality \eqref{RRGeneral} thus yields $h^1\big(F,\cO_F(2A-h)\big)=0$, hence the above sequence must split, again a contradiction. The same argument also excludes the case $(h,3h-A)$.

Let $(c_1,D)=(A,3h-A)$. Since $(2A-3h)^2=(3h-2A)^2=-20$, we deduce that both $h^0\big(F,\cO_F(2A-3h)\big)=0$, and $h^2\big(F,\cO_F(2A-3h)\big)=h^0\big(F,\cO_F(3h-2A)\big)=0$. It follows from Equality \eqref{RRGeneral} that $h^1\big(F,\cO_F(2A-3h)\big)=8$. Since $h^1\big(F,\cO_F(A)\big)=h^2\big(F,\cO_F(A)\big)=0$, it follows $h^1\big(F,\cE\big)=h^1\big(F,\cO_F(2A-3h)\big)\ne0$ from the cohomology of Sequence \eqref{seqIdeal}. 

The same argument can be also used to exclude the cases $(2A,3h-A)$, $(A+h,3h-A)$, $(3h-A,A)$, $(4h-A,A)$, $(6h-2A,A)$.
In the cases $(A+2h,3h-A)$, $(5h-A,A)$ we must modify slightly the above argument by taking the cohomology of Sequence \eqref{seqIdeal} twisted by $\cO_F(-h)$.
\end{proof}

\begin{remark}
We have $h^1\big(F,\cO_F(2A-2h)\big)=h^1\big(F,\cO_F(4h-2A)\big)=6$, thus there exist indecomposable extensions
\begin{gather*}
0\longrightarrow \cO_F(A)\longrightarrow \cE\longrightarrow \cO_F(2h-A)\longrightarrow 0,\\
0\longrightarrow \cO_F(3h-A)\longrightarrow \cE\longrightarrow \cO_F(A-h)\longrightarrow 0.
\end{gather*}
Both $\cO_F(A)$ and $\cO_F(-A)$ are aCM, thus the same is true for the above bundles $\cE$ and it is easy to check that such bundles $\cE$ are initialized as well. As pointed out in \cite{C--N}, Theorem 4.1 and Remark 4.2, we know that such bundles are indecomposable and that each non--zero section of the above bundles vanishes on a divisor in $\vert A\vert$ and $\vert 3h-A\vert$ respectively.
\end{remark}

\section{Proof of the main Theorem}
\label{sMain}
In this section we will prove the main theorem stated in the introduction. We will first give the proof when $c_1$ is not effective, then we will examine the case of effective $c_1$. 

\subsection{The case of non--effective $c_1$.}
If $s\in H^0\big(F,\cE\big)$ and $(s)_0=E\cup D$, thanks to the previous Theorem \ref{tc_1Classification}, we know that $D=0$ and $E$ is a point.

\begin{lemma}
\label{lc_1+h}
Let $\cE$ be an initialized indecomposable aCM bundle of rank $2$ on a general determinantal quartic surface $F\subseteq\p3$.

If $c_1$ is not effective, then $c_1+h$ is effective.
\end{lemma}
\begin{proof}
If $c_1$ is not effective, then $D=0$. Twisting Sequence \eqref{seqIdeal} by $\cO_F(-c_1+th)$, we obtain
\begin{equation}
\label{seqIdealTwisted}
0\longrightarrow \cO_F(D-c_1+th)\longrightarrow \cE^\vee(th)\longrightarrow \cI_{E\vert F}(th-D)\longrightarrow0.
\end{equation}

The cohomology of Sequence \eqref{seqIdealTwisted} with $t=-1$ yields 
$$
h^0\big(F,\cO_F(c_1+h)\big)=h^2\big(F,\cO_F(-c_1-h)\big)\ge h^1\big(F,\cI_{E\vert F}(-h)\big).
$$
We also know that $\deg(E)=1$, hence the cohomology of Sequence \eqref{seqStandard} twisted by $\cO_F(-h)$ yields $h^1\big(F,\cI_{E\vert F}(-h)\big)=1$.
\end{proof}

Assume that $c_1=xh+yA$. It follows from Lemmas \ref{lGG} and \ref{lc_1+h} that either $c_1=-h$, or the inequalities
\begin{gather*}
y^2+3(x+1)y+(x+1)^2\ge1,\qquad 2(x+1)+3y\ge1,\\
y^2-3(6-x)y+(6-x)^2\ge1,\qquad 2(6-x)-3y\ge1,
\end{gather*}
must hold. We obtain a picture similar to Figure 1, with the bottom branch of hyperbola, shifted by $-1$ in the vertical direction. The admissible values of $c_1$ correspond either to $(-1,0)$, or to the integral points $(x,y)$ in the bounded region such that either $y^2+3xy+x^2\le0$ or $2x+3y\le0$, because $c_1$ is not effective. Thus $c_1$ is in the following list:
\begin{equation}
\label{listc_1+h}
-h,\quad 3A-2h,\quad A-h,\quad 2A-h,\quad 2h-A,\quad 5h-2A,\quad 7h-3A.
\end{equation}

\begin{theorem}
\label{tc_1NonEffective}
Let $\cE$ be an initialized indecomposable aCM bundle of rank $2$ on a general determinantal quartic surface $F\subseteq\p3$.

If $c_1$ is not effective, then $c_1\in\{\ -h,A-h,2h-A\ \}$. The zero locus $E$ of a general $s\in H^0\big(F,\cE)$ is a point. Moreover all these cases occur on $F$, and $\cE$ is never Ulrich.
\end{theorem}
\begin{proof}
We know that the above three cases occur thanks to Examples \ref{c_1=0,-1} and \ref{c_1=A-h}.

Looking at List \eqref{listc_1+h} it remains to exclude the cases $3A-2h$, $2A-h$, $5h-2A$, $7h-3A$. Recall that $D=0$ and $\deg(E)=1$.

We have $(3h-2A)^2=-20$, thus $h^0\big(F,\cO_F(3h-2A)\big)=h^1\big(F,\cO_F(3h-2A)\big)=0$, hence $h^1\big(F,\cO_F(3h-2A)\big)=8$ (see Equality \eqref{RRGeneral}). Thus we obtain $h^1\big(F,\cI_{E\vert F}(3h-2A)\big)=9$ from the cohomology of Sequence \eqref{seqStandard} twisted by $\cO_F(3h-2A)$. The cohomology of Sequence \eqref{seqIdeal} gives the exact sequence
$$
0\longrightarrow H^1\big(F,\cE\big)\longrightarrow H^1\big(F,\cI_{E\vert F}(3h-2A)\big)\longrightarrow H^2\big(F,\cO_F\big),
$$
hence $h^1\big(F,\cE)\ge8$, a contradiction. One can argue similarly when $c_1=7h-3A$.

Let $c_1=2A-h$. We trivially have $h^0\big(F,\cI_{E\vert F}(h)\big)=3$ and $h^1\big(F,\cI_{E\vert F}(h)\big)=0$, hence the cohomology of Sequence \eqref{seqIdeal} twisted by $\cO_F(2h-2A)$ gives the exact sequence 
$$
H^0\big(F,\cI_{E\vert F}(h)\big)\longrightarrow H^1\big(F,\cI_{E\vert F}(2h-2A)\big)\longrightarrow H^1\big(F,\cE(h)\big)\longrightarrow 0.
$$
Since $(2h-2A)^2=-16$, as usual we deduce that $h^1\big(F,\cI_{E\vert F}(2h-2A)\big)\ge 6$, hence $h^1\big(F,\cE(h)\big)\ge3$ necessarily, a contradiction. The same argument excludes the case $c_1=5h-2A$.

Finally, notice that if $c_1$ is not effective, then $\cE$ is not globally generated, thus it cannot be Ulrich (see Remark \ref{rUlrich}).
\end{proof}

\begin{remark}
\label{rc_1NonEffective}
Let $\cE$ be as in the statement above. We know that it fits into Sequence \eqref{seqIdeal} with $D=0$. Since
$$
h^1\big(F,\cO_F(h)\big)=h^1\big(F,\cO_F(h-A)\big)=h^1\big(F,\cO_F(A-2h)\big)=0
$$
it follows that such a bundle is necessarily unique once the point is fixed. Thus $\cE$ actually coincides with the bundles constructed in Examples \ref{c_1=0,-1} and \ref{c_1=A-h}.
\end{remark}

\subsection{The case of effective $c_1$.}
If $s\in H^0\big(F,\cE\big)$ and $(s)_0=E\cup D$, thanks to the previous Theorem \ref{tc_1Classification}, we know that $c_1-D$ is effective, unless $c_1=2h$ and $D\ne0$.

Lemma \ref{lGenerator} and Remark \ref{rUlrich} imply that $1\le h^0\big(F,\cE\big)\le8$, because $\cE$ is initialized. Moreover equality holds if and only if $\cE$ is Ulrich. If this is the case $D=0$ for a general choice of $s$ (see again Remark \ref{rUlrich}). 

Assume $D\ne0$. We have $1\le h^0\big(F,\cE\big)\le7$, then the cohomology of Sequence \eqref{seqIdeal} and the effectiveness of $D$ give $h^0\big(F,\cO_F(D)\big)\le7$ and $h^2\big(F,\cO_F(D)\big)=0$. 

Proposition \ref{SD} and Lemma \ref{lD^2ge4} also give $h^1\big(F,\cO_F(D)\big)=0$. It follows from Equality \eqref{RRGeneral} that $D^2\le10$. We deduce that $D^2=4$ by Lemma \ref{lD^2ge4} and Remark \ref{rIntersection}. 

If $D\in\vert xh+yA\vert$, then $x^2+3xy+y^2=1$. Moreover $2x+3y\ge1$ (see Remark \ref{rEffective}), hence 
$$
y=\frac{-3x+\sqrt{5x^2+4}}2.
$$
Taking into account List \eqref{listc_1}, an analysis case by case of the inequality $(c_1-D)h\ge2$ shows that the possible cases with effective $c_1-D$ are exactly the ones listed in the following table.

\begin{gather*}
\text{Table A}\\
\begin{array}{|c|c|c|c|}           \hline
  { c_1 } & {D } & D-c_1+h & \text{ Is $D-c_1+h$ effective? y/n. } \\ \hline
6h & 0,\ A,\ 3h-A & -5h,\ A-5h,\ -2h-A & n,\ n,\ n  \\  \hline
6h-A & 0,\ A,\ 3h-A & A-5h,\ 2A-5h,\ -2h & n,\ n,\ n  \\  \hline
6h-2A & 0,\ 3h-A & 3h-6A,\ A-2h& n,\ n  \\  \hline
5h & 0,\ A,\ 3h-A & -4h,\ A-4h,\ -h-A & n,\ n,\ n  \\  \hline
5h-A & 0,\ 3h-A & A-4h,\ -h & n,\ n  \\  \hline
4h & 0,\ A,\ 3h-A & -3h,\ A-3h,\ -A & n,\ n,\ n  \\  \hline
4h-A & 0,\ 3h-A & A-3h,\ 0 & n,\ y  \\  \hline
3h+A & 0,\ A,\ 3h-A & -2h-A,\ -2h,\ h-2A & n,\ n,\ n  \\  \hline
3h & 0,\ A,\ 3h-A & -2h,\ A-2h,\ h-A & n,\ n,\ n  \\  \hline
3h-A & 0,\ 3h-A & A-2h,\  h & n,\ y  \\  \hline
2h+A & 0,\ A & -h-A,\ -h & n,\ n  \\  \hline
2h & 0 & -h & n  \\  \hline
h+A & 0,\ A & -A,\ 0 & n,\ y \\  \hline
h & 0 & 0 & y  \\  \hline
2A & 0,\ A & h-2A,\ h-A & n,\ n  \\  \hline
A & 0,\ A & h-A,\ h & n,\ y  \\  \hline
0 & 0 & 0 & y  \\  \hline
\end{array}
\end{gather*}
The importance of the two last  columns will be clear in a while.

We know from Equality \eqref{RRGeneral} that
\begin{equation}
\label{chi}
h^0\big(F,\cE\big)+h^2\big(F,\cE\big)=4+\frac{c_1^2}2-c_2,\qquad h^2\big(F,\cE(-h)\big)=8+\frac{c_1^2}2-c_1h-c_2.
\end{equation}

Looking at Table A one deduces from Sequence \eqref{seqIdealTwisted} with $t=0$ that $h^2\big(F,\cE\big)=0$ if $(c_1,D)$ is neither $(3h-A,3h-A)$, nor $(A,A)$, nor $(0,0)$.

Looking at Table A one also deduces from Sequence \eqref{seqIdealTwisted} with $t=1$ that if $D\ne0$, then $h^2\big(F,\cE(-h)\big)=h^0\big(F,\cE^\vee(h)\big)=0$ in all the listed cases but $(c_1,D)$ in the following short list:
\begin{equation}
\label{listD-c_1+h}
(4h-A,3h-A),\quad
(3h-A,3h-A),\quad
(h+A,A),\quad
(A,A),\quad (0,0).
\end{equation}

\begin{proposition}
\label{pD=0}
Let $\cE$ be an initialized indecomposable aCM bundle of rank $2$ on a general determinantal quartic surface $F\subseteq\p3$.

Then exactly one of the following cases occurs:
\begin{enumerate}[label=(\roman{*})]
\item the zero--locus of a general section $s\in H^0\big(F,\cE\big)$ has pure dimension $0$;
\item $c_1=2h$ and the zero locus of each non--zero $s\in H^0\big(F,\cE\big)$ is in either $\vert A\vert$, or $\vert 3h-A\vert$.
\end{enumerate} 
\end{proposition}
\begin{proof}
Assume $D\ne0$, so that $1\le h^0\big(F,\cE\big)\le7$. Let $(c_1,D)$ be not in List \eqref{listD-c_1+h}. By combining the above informations we obtain the equality $h^0\big(F,\cE\big)=c_1h-4$, hence $5\le c_1h\le 11$, which turns out to be impossible by computing $c_1h$ from Table A.

Thus we have to restrict to $(c_1,D)$ in List \eqref{listD-c_1+h}. In the case $(0,0)$ there is nothing to prove. In the case $(A,A)$ we deduce $h^0\big(F,\cE\big)=4$, $h^2\big(F,\cE\big)=h^0\big(F,\cE^\vee\big)=1$, $h^2\big(F,\cE(-h)\big)=h^0\big(F,\cE^\vee(h)\big)=4$  Sequences \eqref{seqIdeal} and \eqref{seqIdealTwisted}. The first Equality \eqref{chi} yields $c_2=1$, the second $c_2=0$, a contradiction. The case $(3h-A,3h-A)$ can be excluded similarly.

In the case $(A+h,A)$ we deduce $h^2\big(F,\cE(-h)\big)=h^0\big(F,\cE^\vee(h)\big)=1$ (from the cohomology of Sequence \eqref{seqIdealTwisted} with $t=1$). The second Equality \eqref{chi} gives $c_2=7$. The cohomology of Sequence \eqref{seqStandard} twisted by $\cO_F(h)$ gives $h^1\big(F,\cI_{E\vert F}(h)\big)\ge3$. Thus the cohomology of Sequence \eqref{seqIdeal} implies $h^1\big(F,\cE\big)\ge3$ because $\cO_F(A)$ is aCM. The case $(4h-A,3h-A)$ can be excluded similarly.
\end{proof}

We are now ready to give the complete classification of the bundles we are interested in when $c_1$ is effective. 

For reader's benefit we recall that a closed subscheme $X\subseteq\p3$ is {\sl aG} (i.e. arithmetically Gorenstein) if it is aCM and $\omega_X\cong \cO_X\otimes\cO_{\p3}(\alpha)$ for a suitable $\alpha\in\bZ$.

Closed subschemes in $\p3$ of degree at most $2$ and dimension $0$ are always aG.
This is no longer true when the degree is $d\ge3$.  Anyhow, if the Hilbert function of the homogeneous coordinate ring $S_E$ of a $0$-dimensional, aG, closed subscheme $E\subseteq\p3$ is $(1,n_1,\dots,n_{\sigma-1},d,\dots)$ with $n_{\sigma-1}<d$, then $E$ is CB with respect to $\cO_{\p3}(\sigma)$ (see \cite{Kr}, Theorem 1.1). 

\begin{theorem}
\label{tc_1Effective}
Let $\cE$ be an initialized indecomposable aCM bundle of rank $2$ on a general determinantal quartic surface $F\subseteq\p3$. 

If $c_1$ is effective, then one of the following holds.
\begin{enumerate}[label=(\roman{*})]
\item $c_1=0$: then $c_2=2$. The zero locus $E$ of a general $s\in H^0\big(F,\cE)$ is a $0$--dimensional aG subscheme contained in exactly one line.
\item $c_1=h$: then $c_2$ is $3$, $4$, or $5$. The zero locus $E$ of a general $s\in H^0\big(F,\cE)$ is a $0$--dimensional aG subscheme of degree $c_2$ of a linear subspace of $\p3$ of dimension $c_2-2$.
\item $c_1\in\{\ A,3h-A\ \}$: then $c_2$  is $3$, or $4$. In this case the zero locus $E$ of a general $s\in H^0\big(F,\cE)$ is $0$--dimensional subscheme which is not aG.
\item $c_1\in\{\ A+h,4A-h\ \}$: then $c_2=8$. In this case the zero locus $E$ of a general $s\in H^0\big(F,\cE)$ is $0$--dimensional subscheme which is not aG and which is contained in exactly one pencil of quadrics.
\item $c_1=2h$: then $c_2=8$. In this case one of the two following cases occur:
\begin{enumerate}[label=(\alph{*})]
\item the zero locus $E$ of a general $s\in H^0\big(F,\cE)$ is a $0$--dimensional aG subscheme which is the base locus of a net of quadrics;
\item the zero locus $E$ of a general $s\in H^0\big(F,\cE)$ is a $0$--dimensional aG subscheme lying on exactly one twisted cubic;
\item the zero locus $E$ of each non--zero $s\in H^0\big(F,\cE)$ is a divisor in either $\vert A\vert$, or $\vert 3h-A\vert$.
\end{enumerate}
\item $c_1=3h$: then $c_2=14$. In this case the zero locus $E$ of a general $s\in H^0\big(F,\cE)$ is a $0$--dimensional aG subscheme which is CB with respect to $\cO_F(3h)$.
\end{enumerate}
Moreover, all the above cases actually occur on $F$, and $\cE$ is Ulrich if and only if $c_1=3h$.
\end{theorem}
\begin{proof}
Recall that the zero--locus $E$ of a general section of $\cE$ has pure codimension $2$, but possibly in the case $c_1=2h$ (see Proposition \ref{pD=0}). With this restriction in mind, we will prove the statement by an analysis case by case with the help of Table A.

We first assume that $h-c_1$ is effective. Looking at "the last column of Table A this is equivalent to $c_1\in\{\ 0,h\ \}$, because $D=0$.

Let $c_1=0$. In this case $\cE\cong\cE^\vee$. Computing the cohomology of Sequence \eqref{seqIdeal} we obtain $h^0\big(F,\cE\big)=h^2\big(F,\cE\big)=1$. The first equality \eqref{chi} thus gives $c_2=2$. 

As already pointed out, $E$ is necessarily aG with Hilbert function $(1,1,\dots,1,2,\dots)$. 
The homogeneous coordinate ring $S_E$ is generated in degree $1$, thus there is only a single $1$ in the sequence: in particular $E$ is contained in exactly two linearly independent planes, hence in a single line.
This case actually occurs as proved in Example \ref{c_1=0,-1}.

Let $c_1=h$. In this case the cohomology of Sequence \eqref{seqIdeal}  implies 
$$
h^0\big(F,\cE\big)= h^0\big(F,\cI_{E\vert F}(h)\big)+1\le h^0\big(F,\cO_F(h)\big)+1=5.
$$
Moreover $h^2\big(F,\cE\big)=h^0\big(F,\cE^\vee\big)=h^0\big(F,\cE(-h)\big)=0$. The first equality \eqref{chi} thus gives
$$
c_2=6-h^0\big(F,\cE\big)=5-h^0\big(F,\cI_{E\vert F}(h)\big).
$$
Thus $c_2\in\{\ 1,2,3,4,5\ \}$. If $c_2=1$ (resp. $2$), we would have $h^0\big(F,\cI_{E\vert F}(h)\big)=4$ (resp. $3$). Nevertheless, in this case $E$ is a point which is necessarily contained in at most three planes (resp. a closed subscheme of degree $2$, which is necessarily contained in at most two planes). Thus we obtain a contradiction.

When $c_2\ge3$, then $E$ is contained in exactly $5-c_2$ linearly independent planes, i.e. in a linear subspace of dimension $c_2-2$ of $\p3$. Since $E$ must be CB with respect to $\cO_F(h)$, it follows that its closed subschemes of degree $c_2-1$ are not contained in smaller linear subspaces, thus $E$ is aG (see e.g. \cite{Sch}, Lemma (4.2)). These cases occur as proved in Example \ref{c_1=1}.

From now on we will assume that $h-c_1$, hence $-c_1$, is not effective. Again we know from the last column of Table $A$ which are the possible values of $c_1$. The cohomology of Sequence \eqref{seqIdealTwisted} with $t=0$ implies 
$h^2\big(F,\cE\big)=h^0\big(F,\cE^\vee\big)=0$.

We know that $h^0\big(F,\cE\big)\le 8$. By combining this restriction with the Equalities \eqref{chi} we obtain the folowing chain of inequalities
\begin{equation}
\label{lBound}
\frac{c_1^2}2-4\le c_2=8+\frac{c_1^2}2-c_1h-h^2\big(F,\cE(-h)\big)\le 8+\frac{c_1^2}2-c_1h,
\end{equation}
whence $c_1h\le12$. Looking at List \eqref{listc_1} one can immediately exclude all the cases but $A$, $h+A$, $2A$, $2h$, $3h-A$, $3h$, $4h-A$, $6h-2A$. We will now examine these cases.

Sequence \eqref{seqIdealTwisted} with $t=1$ gives
$$
H:=h^2\big(F,\cE(-h)\big)=h^0\big(F,\cE^\vee(h)\big)\le h^0\big(F,\cI_{E\vert F}(h)\big)\le4.
$$

Let $c_1=2A$. In this case Inequalities \eqref{lBound} force $c_2=4$, hence $H=0$. We will check that $E$ is the intersection of two divisors in $\vert A\vert$, by checking that $h^0\big(F,\cI_{E\vert F}(A)\big)=2$: indeed $A^2=4$ and the elements of $\vert A\vert$ are integral (see Remark \ref{rIntersection}). 
Notice that the cohomology of Sequence \eqref{seqIdeal} twisted by $\cO_F(-A)$ gives $h^0\big(F,\cI_{E\vert F}(A)\big)=h^0\big(F,\cE(-A)\big)$. Since $h^2\big(F,\cE(-A)\big)=h^0\big(F,\cE(-A)\big)$, it follows that $2h^0\big(F,\cE(-A)\big)=4$, thanks to Equality \eqref{RRGeneral}. 

The above discussion proves the existence of an exact sequence of the form
$$
0\longrightarrow\cO_F\longrightarrow\cO_F(A)^{\oplus2}\longrightarrow\cI_{E\vert F}(2A)\longrightarrow0,
$$
besides Sequence \eqref{seqIdeal}. Thanks to Theorem \ref{tCB} we deduce that $\cE\cong\cO_F(A)^{\oplus2}$, because $h^1\big(F,\cO_F(-2A)\big)=0$. The same argument allows us to exclude the case $c_1=6h-2A$ as well.

Let $c_1=A$ (the case $c_1=3h-A$ can be handled similarly), hence $h^0\big(F,\cE\big)=6-c_2\le6$. Inequalities \eqref{lBound} imply $H\in\{\ 0,1,2,3,4\ \}$ and $c_2=4-H$. The case $H=4$ cannot occur, because $E\ne\emptyset$ (see Lemma \ref{lNonEmpty}). The case $H=3$ cannot occur too: indeed, in this case, $E$ should be a point, thus it cannot be CB with respect to $\cO_F(A)$, because this line bundle is globally generated. 

Assume $H=2$. In this case $E$ is a $0$--dimensional closed subscheme of degree $2$. As already pointed out above we have $h^0\big(F,\cI_{E\vert F}(h)\big)=2$. Thus, computing the cohomology of Sequence \eqref{seqIdealTwisted} for $t=0$ and $1$, one checks that $\cE^\vee(h)$ is initialized. Moreover, it is also aCM and indecomposable (by hypothesis) because the same is true for $\cE$. We have $c_1(\cE^\vee(h))=2h-A$ and $c_2(\cE^\vee(h))=0$. We thus deduce from Theorem \ref{tc_1NonEffective} that this case cannot occur.

The same argument when $H=1$ implies that $\cE^\vee(h)$ is initialized indecomposable aCM with $c_1(\cE^\vee(h))=2h-A$ and $c_2(\cE^\vee(h))=1$. As pointed out at the end of Example \ref{c_1=A-h} we can construct a bundle with all the above properties, thus such an $\cE$ exists. Moreover, one easily checks $h^0\big(F,\cI_{E\vert F}(h)\big)=h^0\big(F,\cE^\vee(h)\big)=1$, thus $E$ is not contained in a line, hence $E$ is not aG in this case.

Let $H=0$. In this case $c_1=A$ and $c_2=4$. We can construct a bundle with these properties, as shown in Example \ref{c_1=A}. We have $h^0\big(F,\cI_{E\vert F}(h)\big)=h^0\big(F,\cE^\vee(h)\big)=0$ as well. Thus $E$ is not contained in any plane, hence its Hilbert function is $(1,4,\dots)$. Since every subscheme of $E$ of degree $3$ trivially lies in a plane, it follows that $E$ is not CB with respect to $\cO_{\p3}(1)$, hence it cannot be aG. 

Let $c_1=h+A$ (the case $c_1=4h-A$ can be handled similarly). Inequalities \eqref{lBound} imply $H\in\{\ 0,1,2\ \}$ and $c_2=8-H$. Consider the exact sequence
$$
0\longrightarrow  \cI_{F\vert \p3}(t)\longrightarrow \cI_{E\vert \p3}(t)\longrightarrow\cI_{E\vert F}(th)\longrightarrow 0
$$
We have that $\cO_F(-A)$ is aCM and $\cI_{F\vert \p3}(t)\cong\cO_F(t-4)$. Thus the cohomologies of Sequence \eqref{seqIdealTwisted} and of the above one for $t=0,1,2$, imply that the Hilbert function of $E$ is $(1,c_2-4,c_2,\dots)$. In particular we obtain $(1,4,8,\dots)$, $(1,3,7,\dots)$, $(1,2,6,\dots)$ if $H=0,1,2$ respectively. 

Since $S_E$ is generated in degree $1$, looking at the growth from degree $1$ to $2$, we deduce that the only admissible Hilbert function is the first one, hence $H=0$. It follows that $c_2=8$ and that $E$ is contained in exactly one pencil of quadrics. In particular, $E$ is not CB with respect to $\cO_{\p3}(2)$, thus $E$ cannot be aG. Moreover, $h^0\big(F,\cE\big)=6$. A bundle with these properties can be constructed as shown at the end of Example \ref{c_1=A}.

Let $c_1=2h$. In this case $h^2\big(F,\cE(-h)\big)=h^0\big(F,\cE(-h)\big)=0$, thus Equality \eqref{RRGeneral} for $\chi(\cE(-h))$ yields $c_2=8$, hence $h^0\big(F,\cE\big)=4$. The complete description and the existence of initialized, indecomposable aCM bundles of rank $2$ on $F$ with $c_1=2h$ and $c_2=8$ is in \cite{C--N}, Proposition 2.1 and Theorems 6.2, 7.1,  4.1, where the three family of bundles corresponding to the cases (a), (b), (c) are described in details.

Notice that in all the above cases we showed that $h^0\big(F,\cE\big)\le7$, hence the considered $\cE$ are not Ulrich.

Let $c_1=3h$. In this case $H=0$ and $c_2=14$, thus $h^0\big(F,\cE\big)=8$, thanks to the first Equality \eqref{chi}: hence $\cE$ is Ulrich. We know from Theorem 1.1 of \cite{C--K--M1} the existence of such kind of bundles on each smooth quartic surface.
\end{proof}

\begin{remark}
Let $\cE$ be as in the statement above. 

As in Remark \ref{rc_1NonEffective} one easily checks that the bundles with $c_1\in \{\ 0,h,2h,3h\ \}$ are uniquely determined by $E$, thus they actually coincide with the bundles constructed in Examples \ref{c_1=0,-1} and \ref{c_1=1}, or with the bundles described in \cite{C--N} and \cite{C--K--M1} respectively.

Similarly, when either $c_1\in\{\ A,3h-A\ \}$ and  $c_2=3$, or $c_1\in\{\ A+h,4h-A\ \}$, then $\cE$ is as at the end of Example either \ref{c_1=A-h}, or \ref{c_1=A}.

Let us consider the case $c_1=A$ more in details (analogous arguments hold in the cases $c_1=3h-A$). The first Equality \eqref{chi} implies $h^0\big(F,\cE\big)=6-c_2$, hence $h^0\big(F,\cI_{E\vert F}(A)\big)=5-c_2\ge1$, thus there is a curve in $C\in \vert A\vert$ containing $E$. If $C$ is smooth, the cohomology of a suitable twist of the exact sequence
$$
0\longrightarrow \cI_{C\vert F}\longrightarrow \cI_{E\vert F}\longrightarrow \cI_{E\vert C}\longrightarrow 0,
$$
and the isomorphisms $\cI_{C\vert F}\cong\cO_F(-C)$, $\cI_{E\vert C}\cong\cO_C(-E)$ yield $h^0\big(C,\cO_C(C-E)\big)=1$. Riemann--Roch theorem on $C$ implies $h^0\big(C,\cO_C(E)\big)=2$. Thus $\cO_C(E)$ is a complete $g^1_{c_2}$ on $C$.  In particular the bundle $\cE$ is constructed exactly as in either Example \ref{c_1=A}, if $c_2=4$, or at the end of Example  \ref{c_1=A-h}, if $c_2=3$.

An analogous construction holds when $c_1=A+h$ (or $4h-A$). Indeed, in this case $c_2=8$. As in the previous case one $h^0\big(F,\cI_{E\vert F}(A)\big)\ge1$. If there is a smooth curve $C\in\vert A+h\vert$, then we can argue as above that $\cO_C(E)$ is a complete $g^1_8$ on $C$. 
\end{remark}

\bigskip
\noindent
Gianfranco Casnati,\\
Dipartimento di Scienze Matematiche, Politecnico di Torino,\\
c.so Duca degli Abruzzi 24, 10129 Torino, Italy\\
e-mail: {\tt gianfranco.casnati@polito.it}

\end{document}